\documentclass[a4paper,11pt]{amsart}

\usepackage[english]{my-shortcuts}
\usepackage{srcltx,algorithm,algorithmic}
\usepackage{graphicx}   

\title[Restricted invertibility near the worst hyperplane]{On the restricted invertibility 
	problem with an additional orthogonality constraint for random matrices}
\author{St\'ephane Chr\'etien} \thanks{National Physical Laboratory, 
	Hampton Road, TW11 0LW, UK. Email: stephane.chretien@npl.co.uk}

\begin{document}
	\maketitle
	

	%
	%
	%

	\begin{abstract}
		The Restricted Invertibility problem is the problem of selecting the largest subset of 
		columns of a given matrix $X$, while keeping the smallest singular value of the extracted submatrix above a certain threshold. 
		In this paper, we address this problem in the simpler case where $X$ is a random matrix but
		with the additional constraint that the selected columns be almost orthogonal to a given vector $v$. Our main result is a lower bound on the number of columns we can extract 
		from a normalized i.i.d. Gaussian matrix for the worst $v$.  
	\end{abstract}
	
	
	\bigskip
	
	\section{Introduction}
	Let $X\in \mathbb R^{n\times p}$. 
	The goal of this short note is to study the following quantity, denoted by 
	$\gamma_{s,\rho_-}(X)$, defined for any $s\le n$ and $\rho_- \in (0,1)$ as 
	\bea 
	\gamma_{s,\rho_-}(X) & = & \sup_{v \in B(0,1)} \inf_{I \subset \mathcal S_{s,\rho_-}} \|X_I^t v\|_{\infty},
	\eea
	where $\mathcal S_{s,\rho_-}(X)$ is the family of all $S$ of $\{1,\ldots,p\}$ with 
	cardinal $|S|=s$, such that $\sigma_{\min}(X_S)\ge \rho_-$. The meaning of the index $\gamma_{s,\rho_-}$ is the following: for any $v\in \R^n$, we look for the "almost orthogonal" family inside the set of columns of $X$ with 
	cardinal $s$, which is the most orthogonal to $v$. 
	
	\subsection{The constrained restricted invertibility problem} 
	
	Once we have an idea of the behavior of $\gamma_{s,\rho_-}(X)$ as a function of 
	$s$, we can derive a lower bound on the number of columns sufficiently 
	orthogonal to a given vector which can be extracted from 
	a given matrix and which form a well conditioned submatrix. This problem 
	is a constrained counterpart to the well known Restricted Invertibility problem 
	which has a long history starting with the seminal work of Bourgain and Tzafriri 
	\cite{BourgainTzafriri:IJM87}. In particular, Bourgain and Tzafriri \cite{BourgainTzafriri:IJM87} obtained the 
	following result for square matrices: 
	\begin{theo}[\cite{BourgainTzafriri:IJM87}]
		\label{BT87}
		Given a $p\times p$ matrix $X$ whose columns have unit $\ell_2$-norm, there exists $I\subset \{1,\ldots,p\}$ 
		with 
		$\d |I|\ge d\frac{p}{\|X\|^2}$ 
		such that  
		$C \le \lb_{\min}(X_I^tX_I) $, 
		where $d$ and $C$ are absolute constants.
	\end{theo}
	See also \cite{Tropp:StudiaMath08} for a simpler proof. 
	Vershynin \cite{Vershynin:IJM01} generalized Bourgain and Tzafriri's result to the case of rectangular matrices 
	and the estimate of $|T|$ was improved. Recently, Spielman and Srivastava proposed in \cite{SpielmanSrivastava:IJM12} 
	a deterministic construction of $T$. Using the same techniques, Youssef \cite{Youssef:IMRN14} was able to improve Vershynin's result.
	
	Applications of such results are well 
	known in the domain of harmonic analysis \cite{BourgainTzafriri:IJM87}. The study of 
	the condition number is also a subject of extensive study in statistics and 
	signal processing \cite{Tropp:CRAS08}. 
	
	In the present paper, we focus on the case where the matrix $X$ is random (which 
	makes the problem a lot easier a priori) but we address 
	restricted invertibility with the additional almost orthogonality constraint
	that $\gamma_{s,\rho_-}(X)$ should be small. Such types of results will find applications in 
	applied harmonic analysis as well. Some applications  
	to sparse recovery and the LASSO procedure in computational statistics are discussed in 
	\cite{ChretienCoherentLasso:Arxiv15}. 
	
	\subsection{Main results}
	\label{MainRes}
	\begin{defi}
		\label{defgam}
		The index $\gamma_{s,\rho_-}(X)$ associated with the matrix $X$ in $\mathbb R^{n\times p}$ is defined by 
		\bea 
		\gamma_{s,\rho_-}(X) & = & \sup_{v \in B(0,1)} \inf_{I \subset \mathcal S_{s,\rho_-}} 
		\|X_I^t v\|_{\infty}.
		\eea
	\end{defi}
	An important remark is that the function $X\mapsto \gamma_{s,\rho_-}(X)$ is nonincreasing in the sense that if 
	we set $X^{\prime\prime}=[X,X^\prime]$, 
	where $X^\prime$ is a matrix in $\mathbb R^{n \times p^\prime}$, then $\gamma_{s,\rho_-}(X)\ge \gamma_{s,\rho_-}(X^\prime)$.
	
	The quantity $\gamma_{s,\rho_-}(X)$ is very small for 
	$p$ sufficiently large, at least for random matrices such as normalized standard Gaussian matrices as shown in the following theorem. 
	\begin{theo}
		\label{gamtheo}
		Assume that $X$ is random matrix in $\mathbb R^{n\times p}$ with i.i.d. columns with 
		uniform distribution on the unit sphere of $\R^n$. 
		Let $\rho_-$ and $\epsilon\in(0,1)$, $C_\kappa\in (0,+\infty)$ and assume that $p \ge \lceil e^{\frac6{\sqrt{2\pi}}}\rceil$. 
		Set 
		\bean 
		K_\epsilon & = & \frac{\sqrt{2\pi}}{6}\left(\left(1+C_\kappa \right) 
		\log\left(1+\frac2{\epsilon}\right)+C_\kappa+\log\left(\frac{C_\kappa}{4}\right)\right).
		\eean 
		Assume that $n$, $\kappa$ and $s$ satisfy 
		\bea
		& n \ge 6, \\
		\nonumber \\
		\label{kappa}
		& \kappa = \max\left\{ 4 e^{-2(\ln(2)-1)},\frac{4 e^3}{(1-\rho_-)^2} \ \left( \frac{(1+K_\epsilon)(1+C_{\kappa})}{c(1-\epsilon)^4}\right)^2 \log^2(p) \log(C_\kappa n)\right\}, \\
		\nonumber \\
		\label{n}
		& \displaystyle{\frac{\max\left\{\kappa s, 2\times 36 \times 3 \times 3,\exp((1-\rho_-)/2) \right)}{C_{\kappa}}} 
		\le n \le \min \left\{\left(\frac{p}{\log(p)}\right)^2,\frac{\exp\left(\frac{1-\rho_-}{\sqrt{2}}p \right)}{C_\kappa}\right\}.
		\eea 
		Then, we have 
		\bea 
		\gamma_{s,\rho_-}(X) & \le & 80\ \frac{\log(p)}{p}
		\eea
		with probability at least $1-5 \ \frac{n}{p\ \log(p)^{n-1}}-9\ p^{-n}$. 
	\end{theo}
	\begin{cor}
		We can take $s$ as large as 
		\bean 
		\label{s}
		\left\lfloor C_{s} \ \frac{ n}{ \log^2(p) \log(C_\kappa n)} \right\rfloor
		\eean 
		with 
		\bean 
		C_{s} & = & \frac{c^2(1-\rho_-)^2(1-\epsilon)^8}{4e^3} \ \frac{C_{\kappa}}{(1+K_{\epsilon})^2(1+C_{\kappa})^2}.
		\eean 
	\end{cor}
	\begin{proof}
		Notice that the constraints (\ref{kappa}) and (\ref{n}) together imply the following constraint on $s$:
		\bean
		& s \le & C_{s} \ \frac{ n}{ \log^2(p) \log(C_\kappa n)}
		\eean 
		The result follows immediately.
	\end{proof}
	
	\section{Proof of Proposition \ref{gamtheo}} 
	\label{Pf}

	\subsection{Constructing an outer approximation for $I$ in the definition of $\gamma_{s,\rho_-}$}
	
	Take $v\in \mathbb R^n$. We construct an outer approximation $\tilde{I}$ 
	of $I$ into which we be able to extract the set $I$. We procede recursively as follows: 
	until $|\tilde{I}|=\min \{\kappa s,p/2\}$, for some positive real number $\kappa$ to be specified later, do 
	\begin{itemize}
		\item Choose $j_1={\rm argmin}_{j=1,\ldots,p} |\la X_{j},v\ra|$ and set $\tilde{I}=\{j_1\}$
		\item Choose $j_2={\rm argmin}_{j=1,\ldots,p, \ j\not\in \tilde{I}} |\la X_{j},v\ra|$ and set $\tilde{I}=\tilde{I}\cup\{j_2\}$
		\item $\cdots$
		\item Choose $j_{k}={\rm argmin}_{j=1,\ldots,p, \ j\not\in \tilde{I}} |\la X_{j},v\ra|$ and set $\tilde{I}=\tilde{I}\cup\{j_k\}$. 
	\end{itemize}

	\subsection{An upper bound on $\|X^t_{\tilde{I}}v\|_\infty$}
	If we denote by $Z_j$ the quantity $|\la X_{j},v\ra|$ and by $Z_{(r)}$ the $r^{th}$ 
	order statistic, we get that 
	\bean 
	\|X^t_{\tilde{I}}v\|_\infty & = & Z_{(\kappa s)}.
	\eean
	Since the $X_j$'s are assumed to be i.i.d. with uniform distribution on the unit sphere of $\R^n$, 
	we obtain that the distribution of $Z_{(r)}$ is the distribution of the $r^{th}$ order statistics
	of the sequence $|X_j^tv|$, $j=1,\ldots,p$. 
	By (5) p.147 \cite{Muirhead:AspectsMultAnal05}, $|X_j^tv|$ has density $g$ and CDF $G$ 
	given by 
	\bean
	g(z) & = \frac1{\sqrt{\pi}} \frac{\Gamma\left(\frac{n}2\right)}{\Gamma\left(\frac{n-1}2\right)}
	\left(1-z^2 \right)^{\frac{n-3}2}
	\textrm{ and }
	G(z) = & 2\ \int_{0}^z g(\zeta) \ d\zeta.
	\eean
	Thus,
	\bean
	F_{Z_{(r)}}(z) & = & \bP \left(B\ge r \right)
	\eean
	where $B$ is a binomial variable $\cal B\left(p,G(z)\right)$.
	Our next goal is to find the smallest value $z_0$ of $z$ which satisfies 
	\bea
	\label{devorder}
	F_{Z_{(\kappa s)}}(z_0) & \ge & 1-p^{-n}.
	\eea
	We have the following standard concentration bound for $B$ (e.g. \cite{Dubhashi:CUP09}): 
	\bean 
	\mathbb P \left(B \le (1-\epsilon) \mathbb E[B] \right) & \le &  \exp\left(-\frac12 \ \epsilon^2  \mathbb E[B]\right)
	\eean 
	which gives 
	\bean 
	\mathbb P \left(B \ge (1-\epsilon) pG(z) \right) & \ge &  1-\exp\left(-\frac12 \ \epsilon^2 p G(z) \right)
	\eean 
	We thus have to look for a root (or at least an upper bound to a root) of the equation 
	\bean 
	G(z) & = & \frac1{\frac12 \ \epsilon^2 } \ \frac{n}{p} \ \log(p). 
	\eean 
	Notice that
	\bean
	G(z) & = & 2\ \frac1{\sqrt{\pi}} \frac{\Gamma\left(\frac{n}2\right)}{\Gamma\left(\frac{n-1}2\right)}
	\ \int_{0}^z  \ \left(1-\zeta^2 \right)^{\frac{n-3}2} \ d\zeta, \\
	& \ge & \frac1{\sqrt{\pi}} \frac{\Gamma\left(\frac{n}2\right)}{\Gamma\left(\frac{n-1}2\right)} \ z
	\eean
	for $z\le 1/\sqrt{2}$. By a straightforward application of Stirling's formula (see e.g. (1.4) 
	in \cite{Qi:JIA10}), we obtain 
	\bean 
	\frac{\Gamma\left(\frac{n}2\right)}{\Gamma\left(\frac{n-1}2\right)} & \ge & \frac{e^{2\ln(2)}}{2} \ \frac{(n-3)^{3/2}}{(n-2)^{1/2}}. 
	\eean 
	Thus, any choice of $z_0$ satisfying
	\bea
	\label{left} 
	z_0 & \ge &  \frac{2\ \sqrt{\pi}}{e^{2\ln(2)}} \ \frac{(n-2)^{1/2}}{(n-3)^{3/2}}\ \frac1{\frac12 \ \epsilon^2} \ 
	\frac{n}{p} \ \log(p)
	\eea
	is an upper bound to the quantile for $(1-\epsilon) p G(z_0)$-order statistics at level $p^{-n}$. 
	We now want to enforce the constraint that 
	\bean
	(1-\epsilon) p G(z_0) & \le & \kappa s. 
	\eean 
	By again a straightforward application of Stirling's formula, we obtain 
	\bean
	G(z) & \le & \frac1{\sqrt{\pi}}\ \frac{ e^2}{2} \ \frac{(n-3)^{3/2}}{(n-2)^{1/2}}\ z
	\eean 
	for $n\ge 4$. Thus, we need to impose that 
	\bea
	\label{right}
	z_0 & \le & \frac{2 \sqrt{\pi}}{e^2} \ \frac{(n-2)^{1/2}}{(n-3)^{3/2}} \ \frac{\kappa s}{(1-\epsilon) p}. 
	\eea 
	Notice that the constraints (\ref{left}) and (\ref{right}) are compatible if 
	\bean 
	\kappa  & \ge & \frac{4}{e^{2(\ln(2)-1)}} \ \frac{1-\epsilon}{\epsilon^2} \ 
	\frac{n}{s} \ \log(p) .
	\eean 
	Take $\epsilon=1-\frac1{n/s \log(p)}$ and obtain
	\bean
	\label{quantunifsph0} 
	\bP \left( \|X^t_{\tilde{I}}v\|_\infty \ge \frac{8\ \sqrt{\pi}}{e^{2\ln(2)}} \ \frac{(n-2)^{1/2}}{(n-3)^{3/2}} \ \frac{n}{p} \ \log(p) \right) & \le & p^{-n}
	\eean
	for 
	\bean 
	\kappa & = & \frac{4}{e^{2(\ln(2)-1)}}  
	\eean 
	for any $p$ such that $n/s \log(p)\ge \sqrt{2}$, which is clearly the case as soon as $p\ge e^{\frac6{\sqrt{2\pi}}}$
	for $s\le n$ as assumed in the proposition. 
	
	If $n\ge 6$, we can simplify (\ref{quantunifsph0}) with 
	\bea
	\label{quantunifsph} 
	\bP \left( \|X^t_{\tilde{I}}v\|_\infty \ge 80 \ \frac{\log(p)}{p} \right) & \le & p^{-n}
	\eea

	\subsection{Extracting a well conditionned submatrix of $X_{\tilde{I}}$}
	The method for extracting $X_I$ from $X_{\tilde{I}}$ uses random column selection. For this 
	purpose, we will need to control the coherence and the norm of $X_{\tilde{I}}$.
	
	{\bf Step 1: The coherence of $X_{\tilde{I}}$}. Let us define the spherical cap 
	\bean 
	\cal C(v,h) & = & \left\{w \in \R^n \mid \la v,w\ra \ge h \right\}.
	\eean 
	The area of $\cal C(v,h)$ is given by 
	\bean 
	Area\left(\cal C(v,h)\right) & = &  Area (\cal S(0,1))  \int_0^{2h-h^2} t^{\frac{n-1}2}(1-t)^{\frac12} dt.
	\eean
	Thus, the probability that a random vector $w$ with Haar measure on the unit sphere $\cal S(0,1)$ falls into 
	the spherical cap $\cal C(v,h)$ is given by 
	\bean 
	\bP \left(w\in \cal C(v,h) \right) & = & \frac{\cal C(v,h)}{\cal S(0,1)} \\
	& = & \frac{\int_0^{2h-h^2} t^{\frac{n-1}2}(1-t)^{\frac12} dt}{\int_0^{1} t^{\frac{n-1}2}(1-t)^{\frac12} dt}. 
	\eean 
	The last term is the CDF of the Beta distribution. Using the fact that 
	\bean 
	\bP \left(X_j \in \cal C(X_{j^\prime},h) \right) & = & \bP \left(X_{j^\prime} \in \cal C(X_j,h) \right)
	\eean 
	the union bound, and the independence of the $X_j$'s, the probability that $X_j \in \cal C(X_{j^\prime},h)$ for some 
	$(j,j^\prime)$ in $\left\{1,\ldots,p\right\}^2$ can be bounded as follows
	\bean 
	\bP \left(\cup_{j\neq j^\prime=1}^{p} \left\{ X_j \in \cal C(X_{j^\prime},h) \right\} \right) & = & \bP \left(\cup_{j<j^\prime=1}^{p} \left\{ X_j \in \cal C(X_{j^\prime},h) \right\} \right) \\
	& \le & \sum_{j<j^\prime=1}^{p} \bP \left(\left\{ X_j \in \cal C(X_{j^\prime},h) \right\} \right) \\
	& = & \sum_{j<j^\prime=1}^{p} \bE \left[ \bP\left(\left\{ X_j \in \cal C(X_{j^\prime},h) \right\} \mid X_{j^\prime}\right)\right] \\
	& = & \frac{p(p-1)}2 \int_0^{2h-h^2} t^{\frac{n-1}2}(1-t)^{\frac12} dt.
	\eean  
	Our next task is to choose $h$ so that 
	\bean 
	\frac{p(p-1)}2 \int_0^{2h-h^2} t^{\frac{n-1}2}(1-t)^{\frac12} dt & \le & p^{-n}.
	\eean  
	Let us make the following crude approximation 
	\bean 
	\frac{p(p-1)}2 \int_0^{2h-h^2} t^{\frac{n-1}2}(1-t)^{\frac12} dt & \le & 
	\frac{p^2}2  (2h)^{\frac{n-1}2} (2h-0).
	\eean  
	Thus, taking 
	\bean 
	h & \ge & \frac12 \ \exp\left( - 2\ \left(\log(p)+\frac{\log\left(p)-\log(2)\right)}{n+1}\right)\right) 
	\eean
	will work. Moreover, since $p\ge2$, we deduce that 
	\bea
	\label{mutilde} 
	\mu (X_{\tilde{I}}) & \le & \frac12 \ p^{-2}
	\eea 
	with probability at least $1-p^{-n}$.
	
	{\bf Step 2: The norm of $X_{\tilde{I}}$}.
	The norm of any submatrix $X_{S}$ with $n$ rows and $\kappa s$ columns of $X$ has the following 
	variational representation
	\bean 
	\|X_S\| & = & \max_{\stackrel{v \in \mathbb R^{n},\ \|v\|=1}{w \in \mathbb R^{\kappa s},\ \|w\|=1}}\ v^t X_S w.
	\eean
	We will use an easy $\epsilon$-net argument to control this norm. For any $v\in \mathbb R^n$, 
	$v^t X_j$, $j\in S$ is a sub-Gaussian random variable satisfying
	\bean 
	\bP \left( |v^t X_j| \ge u \right) & \le & 2 \exp\left(-c n \ u^2 \right),
	\eean
	for some constant $c$. Therefore, using the fact that $\|w\|=1$, we have that 
	\bean 
	\bP \left( \left|\sum_{j\in S} v^t X_Sw \right| \ge u \right) & \le & 2 \exp\left(-c n \ u^2 \right).
	\eean
	Let us recall two useful results of Rudelson and Vershynin. The first one gives a bound on the 
	covering number of spheres. 
	\begin{prop} {\rm (\cite[Proposition 2.1]{RudelsonVershynin:CPAM09})}.
		\label{net}
		For any positive integer $d$, there exists an $\epsilon$-net of the unit sphere of $\R^d$ of cardinality 
		\bean 
		2d \left(1+\frac2{\epsilon}\right)^{d-1}.
		\eean
	\end{prop}
	The second controls the approximation of the norm based on an $\epsilon$-net. 
	\begin{prop} {\rm (\cite[Proposition 2.2]{RudelsonVershynin:CPAM09})}.
		\label{rv2}
		Let $\mathcal N$ be an $\epsilon$-net of the unit sphere of $\R^d$ and let 
		$\mathcal N^\prime$ be an $\epsilon^\prime$-net of the unit sphere of $\R^{d^\prime}$. Then 
		for any linear operator $A: \R^d \mapsto \R^{d^\prime}$, we have 
		\bean 
		\|A\| & \le & \frac{1}{(1-\epsilon)(1-\epsilon^\prime)} \sup_{\stackrel{v \in \mathcal N}{w \in \mathcal N^\prime}}
		|v^t Aw|. 
		\eean   
	\end{prop}
	Let $\mathcal N$ (resp. $\mathcal N^\prime$) be an $\epsilon$-net of the unit sphere of $\R^{\kappa s}$ 
	(resp. of $\R^n$). On the other hand, we have that 
	\bean 
	\bP \left( \sup_{\stackrel{v \in \mathcal N}{w \in \mathcal N^\prime}}
	|v^t Aw|\ge u \right) & \le & 2|\mathcal N||\mathcal N^\prime| \exp\left(-c n \ u^2 \right), \\
	& \le & 8 \ n \kappa s \left(1+\frac2{\epsilon}\right)^{n+\kappa s-2} \exp\left(-c n \ u^2 \right),
	\eean 
	which gives 
	\bean 
	\bP \left( \sup_{\stackrel{v \in \mathcal N}{w \in \mathcal N^\prime}}
	|v^t Aw|\ge u \right) & \le & 8 \ \frac{n \kappa s \ \epsilon^2}{(2+\epsilon)^2} 
	\exp\left(-\left(c n \ u^2 -(n+\kappa s) \log\left(1+\frac2{\epsilon}\right)\right)  \right).
	\eean 
	Using Proposition (\ref{rv2}), we obtain that 
	\bean
	\bP \left( \|X_S\| \ge u \right) & \le & \bP \left( \frac{1}{(1-\epsilon)^2} 
	\sup_{\stackrel{v \in \mathcal N}{w \in \mathcal N^\prime}} |v^t Aw| \ge u  \right).
	\eean
	Thus, we obtain 
	\bean
	\bP \left( \|X_S\| \ge u \right) & \le & 8 \ \frac{n \kappa s \ \epsilon^2}{(2+\epsilon)^2} 
	\exp\left(-\left(c n \ (1-\epsilon)^4 \ u^2-(n+\kappa s) \log\left(1+\frac2{\epsilon}\right)\right)  \right).
	\eean
	To conclude, let us note that 
	\bean 
	\bP\left(\|X_{\tilde{I}}\|\ge u \right) & \le & \bP \left( \max_{\stackrel{S \subset \{1,\ldots,p\}}{|S|=\kappa s}} \|X_S\| \ge u \right) \\
	& \le & {p\choose {\kappa s}}\ 8 \ \frac{n \kappa s \ \epsilon^2}{(2+\epsilon)^2} 
	\exp\left(-\left(c n \ (1-\epsilon)^4 \ u^2-(n+\kappa s) \log\left(1+\frac2{\epsilon}\right)\right)\right).
	\eean 
	and using the fact that 
	\bean 
	{p\choose {\kappa s}} & \le & \left(\frac{e \ p}{\kappa s}\right)^{\kappa s},
	\eean
	one finally obtains
	\bean 
	\bP\left(\|X_{\tilde{I}}\|\ge u \right) & \le & 8 \  
	\exp\left(-\left(c n \ (1-\epsilon)^4 \ u^2-(n+\kappa s) \log\left(1+\frac2{\epsilon}\right)- 
	\kappa s \log\left(\frac{e \ p}{\kappa s}\right) -\log\left(\frac{n \kappa s \ \epsilon^2}{(2+\epsilon)^2}\right)\right)\right).
	\eean 
	The right hand side term will be less than $8 p^{-n}$ when 
	\bean 
	n \log(p) & \le & c n \ (1-\epsilon)^4 \ u^2-(n+\kappa s) \log\left(1+\frac2{\epsilon}\right)- 
	\kappa s \log\left(\frac{e \ p}{\kappa s}\right) -\log\left(\frac{n \kappa s \ \epsilon^2}{(2+\epsilon)^2}\right).
	\eean
	This happens if 
	\bean 
	u^2 & \ge & \frac1{c (1-\epsilon)^4} \left(n \frac{\log(p)}{n}+\left(1+\frac{\kappa s}{n}\right) 
	\log\left(1+\frac2{\epsilon}\right)+ \frac{\kappa s}{n} 
	\log\left(\frac{e \ p}{\kappa s}\right)+\frac1{n}\log\left(\frac{n \kappa s \ \epsilon^2}{(2+\epsilon)^2}\right)\right).
	\eean
	Notice that 
	\bea
	\label{sn} 
	& &\left(1+\frac{\kappa s}{n}\right) 
	\log\left(1+\frac2{\epsilon}\right)+\frac{\kappa s}{n} 
	\log\left(\frac{e}{\kappa s}\right)
	+\frac1{n}\log\left(\frac{n \kappa s \ \epsilon^2}{(2+\epsilon)^2}\right) \\
	& & \hspace{3cm} \le 
	\left(1+C_\kappa \right) 
	\log\left(1+\frac2{\epsilon}\right)+C_\kappa+\frac1{n}\log\left(\frac{C_\kappa n^2}{4}\right), \nonumber \\
	& & \hspace{3cm} \le K_{\epsilon} \ \frac{6}{\sqrt{2\pi}}, \nonumber
	\eea 
	since $n\ge 1$. Now, since
	\bean 
	\frac{6}{\sqrt{2\pi}} & \le \log(p) \le & \frac{n+\kappa s}{n} \log(p),
	\eean 
	we finally obtain 
	\bea
	\label{normbnd}
	\bP\left(\|X_{\tilde{I}}\|\ge \frac{1+K_\epsilon} 
	{c (1-\epsilon)^4}\frac{n+\kappa s}{n} \log(p) \right) & \le &  \frac8{p^{n}}.
	\eea

	{\bf Step 3}. We will use the following lemma on the distance to identity of randomly 
	selected submatrices. 
	\begin{lemm} 
		\label{submat}
		Let $r\in(0,1)$. Let $n$, $\kappa$ and $s$ satisfy conditions (\ref{n}) and (\ref{kappa}) assumed in 
		Proposition \ref{gamtheo}.
		Let $\Sigma\subset \left\{1,\ldots,\kappa s\right\}$ 
		be a random support with uniform distribution on index sets with cardinal $s$.
		Then, with probability greater than or equal to $1-9 \ p^{-n}$ on $X$, the following bound holds:
		\bea \label{sing}
		\bP \left(\|X^t_{\Sigma}X_{\Sigma}-\Id_{s} \|\ge r \mid X\right) & < & 1.
		\eea
	\end{lemm}
	\begin{proof}
		See Appendix.
	\end{proof}

	Taking $r=1-\rho_-$, we conclude from Lemma \ref{submat} that, for any $s$ satisfying (\ref{s}), 
	there exists a subset $\tilde{\tilde{I}}$ of $\tilde{I}$ with cardinal $s$ such that 
	\bean 
	\sigma_{\min}\left(X_{\tilde{\tilde{I}}}\right)\ge \rho_-.
	\eean
	
	\subsubsection{The supremum over an $\epsilon$-net}
	Recalling Proposition \ref{net}, there exists an $\epsilon$-net $\cal N$ covering the unit sphere in $\R^n$ with cardinal 
	\bean 
	|\mathcal N | & \le & 2n\left( 1+\frac2{\epsilon}\right)^{n-1}.
	\eean
	Combining this with (\ref{quantunifsph}), we have that 
	\bea
	\nonumber
	& \bP \left( 
	\sup_{v\in \cal N} \inf_{I\subset \cal S_{s,\rho_-}} \left\|X_I^tv \right\| \ge 
	\frac{8\ \sqrt{\pi}}{e^{2\ln(2)}} \ \frac{n\ (n-2)^{1/2}}{(n-3)^{3/2}} \ \frac{\log(p)}{p} \right) \\ 
	\label{netcontrol}
	& \le 2n\left( 1+\frac2{\epsilon}\right)^{n-1} \ p^{-n}+9\ p^{-n}.
	\eea
	
	\subsection{From the $\epsilon$-net to the whole sphere}
	
	For any $v^\prime$, one can find $v\in \cal N$ with $\|v^\prime-v\|_2\le \epsilon$. Thus, we have 
	\bea
	\nonumber
	\|X_I^tv^\prime \|_{\infty} & \le & \|X_I^t v \|_{\infty}+ \|X_I^t (v^\prime-v) \|_{\infty} \\
	\nonumber 
	& \le & \|X_I^t v \|_{\infty}+ \max_{j\in I} |\la X_j,(v^\prime-v) \ra |\\
	\nonumber 
	& \le & \|X_I^t v \|_{\infty}+ \max_{j\in I} \| X_j\|_{2} \|v^\prime-v\|_2 \\
	\label{sph} & \le & \|X_I^t v \|_{\infty}+ \epsilon.
	\eea
	Taking 
	\bean 
	\epsilon & = & 80\ \frac{\log(p)}{p},
	\eean 
	we obtain from (\ref{sph}) and (\ref{netcontrol}) that 
	\bean 
	& \bP \left( 
	\sup_{\|v\|_2=1} \inf_{I\subset \cal S_{s,\rho_-}} \left\|X_I^tv \right\| \ge 
	80\ \frac{\log(p)}{p}\right) \\ & \le 20 \ n\left( 1+\frac{p}{80\log(p)}\right)^{n-1} \ p^{-n}+9\ p^{-n}
	\eean 
	and thus,  
	\bean 
	& \bP \left(\sup_{\|v\|_2=1} \inf_{I\subset \cal S_{s,\rho_-}} \left\|X_I^tv \right\| \ge 
	80\ \frac{\log(p)}{p}\right) \\ & \le 5 \ \frac{n}{p\ \log(p)^{n-1}}+9\ p^{-n},
	\eean
	for $p\ge \exp(6/\sqrt{2\pi})$.

	\appendix 
	
	\section{Proof of Lemma \ref{submat}}
	For any index set $S\subset \{1,\ldots,\kappa s\}$ with cardinal $s$, 
	define $R_S$ as the diagonal matrix with 
	\bean 
	(R_{S})_{i,i} & = & 
	\begin{cases}
		1 \textrm{ if } i\in S, \\
		0 \textrm{ otherwise.}
	\end{cases}
	\eean
	Notice that we have 
	\bean 
	\left\|X_S^tX_S-I\right\| & = & \left\|R_SHR_S\right\| 
	\eean
	with $H=X^tX-I$. In what follows, $R_\delta$ simply denotes a diagonal 
	matrix with i.i.d. diagonal components $\delta_j$, $j=1,\ldots,\kappa s$ 
	with Bernoulli $B(1,1/\kappa)$ distribution. Let $R^{\prime}$ be an independent copy of $R$.
	Assume that $S$ is drawn uniformly at random among index sets of $\{1,\ldots,\kappa s\}$
	with cardinal $s$. By an easy Poissonization argument, similar to \cite[Claim $(3.29)$ p.2173]{CandesPlan:AnnStat09},
	we have that 
	\beq \label{poisse}
	\bP \left(\|R_sHR_s\|\ge r \right) \ \le \ 2\ \bP \left(\|RHR\|\ge r\right),
	\eeq
	and by Proposition 4.1 in \cite{ChretienDarses:SPL12}, we have that  
	\bea \label{dec}
	\bP\left(\|RHR\|\ge r\right) & \le & 36\ \bP\left(\|RHR^\prime\|\ge r/2 \right).
	\eea
	In order to bound the right hand side term, we will use \cite[Proposition 4.2]{ChretienDarses:SPL12}. 
	Set $r^\prime=r/2$. Assuming that $\kappa \frac{{r^\prime}^2}{e} \ge u^2 \ge \frac{1}{\kappa}\|X\|^4$ and 
	$v^2\ge \frac{1}{\kappa}\|X\|^2$, the right hand side term can be bounded from above as follows: 
	\bea\label{inv_bound}
	\bP \left(\|RHR'\|\ge r^\prime \right) & \le & 3 \ \kappa s\ \mathcal  V(s,[r^\prime,u,v]),
	\eea
	with
	\bean
	\mathcal  V(s,[r^\prime ,u,v]) & = & \left(e\frac{1}{\kappa} \frac{u^2}{{r^\prime}^2} \right)^{\frac{{r^\prime}^2}{v^2}} 
	+\left(e \frac{1}{\kappa}\frac{\|M\|^4}{u^2} \right)^{u^2/\|M\|^2} 
	+\left(e \frac{1}{\kappa}\frac{\|M\|^2}{v^2} \right)^{v^2/\mu(M)^2}.
	\eean
	Using (\ref{mutilde}) and (\ref{normbnd}), we deduce that with probability at least 
	$1-8 p^{-n}-p^{-n}$, we have 
	\bean
	\mathcal  V(s,[r^\prime,u,v]) & = & \left(e\frac{1}{\kappa} \frac{u^2}{{r^\prime}^2} \right)^{\frac{{r^\prime}^2}{v^2}} 
	+\left(e \frac{1}{\kappa}\frac{\left(\frac{1+K_\epsilon}{c(1-\epsilon)^4}\frac{n+\kappa s}{n} \log(p)\right)^4}{u^2} \right)
	^{\frac{u^2}{\left(\frac{1+K_\epsilon}{c(1-\epsilon)^4}\frac{n+\kappa s}{n} \log(p)\right)^2}}  \\
	& & \hspace{.5cm} +\left(e \frac{1}{\kappa}
	\frac{\left(\frac{1+K_\epsilon}{c(1-\epsilon)^4}\frac{n+\kappa s}{n} \log(p)\right)^2}{v^2} \right)^{\frac{v^2}{\frac12 \ p^{-2}}}.
	\eean
	Take $\kappa$, $u$ and $v$ such that 
	\bean 
	v^2 & = &  {r^\prime}^2 \ \frac1{\log(C_\kappa\ n)}  \\
	\\
	u^2 & = &  C_{\cal V}\ \left( \frac{1+K_\epsilon}{c(1-\epsilon)^4}\frac{n+\kappa s}{n} \log(p)\right)^2,\\
	\\
	\kappa & \ge & e^3 \ \frac{C_{\cal V}}{{r^\prime}^2} \ \left( \frac{1+K_\epsilon}{c(1-\epsilon)^4}\frac{n+\kappa s}{n} \log(p)\right)^2 \\
	\eean
	for some $C_{\cal V}$ possibly depending on $s$. Since $\kappa s\le C_{\kappa} n$,
	this implies in particular that 
	\bea 
	\label{kap}
	\kappa & \ge & e^3 \ \frac{C_{\cal V}}{{r^\prime}^2} \ \left( \frac{(1+K_\epsilon)(1+C_{\kappa})}{c(1-\epsilon)^4} \log(p)\right)^2.
	\eea
	Thus, we obtain that 
	\bean
	\mathcal  V(s,[r^\prime,u,v]) & = & \left(\frac{1}{e^2} \right)^{\log(C_\kappa n)} 
	+\left(\frac{{r^\prime}^2}{e^2 \ C_{\cal V}^2} \right)^{C_{\cal V}}
	+\left(\frac{\log(C_\kappa n)}{e^2\ C_{\cal V}} \right)^{\frac{2 {r^\prime}^2\ p^2}{\log(C_\kappa n)}}.
	\eean
	Using (\ref{poisse}), (\ref{dec}) and (\ref{inv_bound}), we obtain that 
	\bean 
	\bP \left(\|R_sHR_s\|\ge r^\prime \right) & \le & 2\times 36 \times 3 \times \kappa s 
	\left(\left(\frac{1}{e^2} \right)^{\log(C_\kappa n)} 
	+\left(\frac{{r^\prime}^2}{e^2\ C_{\cal V}^2} \right)^{C_{\cal V}}
	+\left(\frac{\log(C_\kappa n)}{e^2\ C_{\cal V}} \right)^{\frac{2 {r^\prime}^2\ p^2}{\log(C_\kappa n)}}\right).
	\eean 
	Take 
	\bea
	\label{CV}
	C_{\cal V}  & = & \log(C_\kappa n)
	\eea
	and, since $p>1$ and $r \in (0,1)$, we obtain 
	\bea
	& & \bP \left(\|R_sHR_s\|\ge r^\prime \right) \nonumber \\
	\label{probe}
	& & \hspace{2cm} \le 2\times 36 \times 3 \times \kappa s 
	\left(\left(\frac{1}{e^2} \right)^{\log(C_\kappa n)} 
	+\left(\frac{{r^\prime}^2}{e^2 \ \log^2(C_\kappa n)} \right)^{\log(C_\kappa n)}+\left(\frac{1}{e^2} \right)^{\frac{2 {r^\prime}^2\ p^2}{\log(C_\kappa n)}}\right).
	\eea 
	Replace $r^\prime$ by $r/2$. Since it is assumed that $n \ge \exp(r/2)/C_{\kappa}$ and $p\ge \sqrt{2} \log(C_{\kappa}n)/r$, 
	it is sufficient to impose that 
	\bean 
	C_\kappa^2 n^2  & \ge & \left(2\times 36 \times 3 \times \kappa s\times 3\right)^{\frac1{\log(e^2)}},
	\eean
	in order for the right hand side of (\ref{probe}) to be less than one. 
	Since $\kappa s\le C_{\kappa}n$, it is sufficient to impose that 
	\bean 
	C_\kappa^2 n^2  & \ge & 2\times 36 \times 3 \times \ C_\kappa n\times 3,
	\eean 
	or equivalently, 
	\bean 
	C_\kappa n  & \ge & 2\times 36 \times 3 \times 3.
	\eean 
	This is implied by (\ref{n}) in the assumptions. 
	On the other hand, combining (\ref{kap}) and (\ref{CV}) implies that one can take 
	\bean 
	\kappa & = & \frac{4e^3}{r^2} \ \left( \frac{(1+K_\epsilon)(1+C_{\kappa})}{c(1-\epsilon)^4}\right)^2 \log^2(p) \log(C_\kappa n),
	\eean 
	which is nothing but (\ref{kappa}) in the assumptions.

\end{document}